\documentclass[12pt]{article}
\usepackage{amssymb}

\usepackage{amsmath,amsthm}
\providecommand{\U}[1]{\protect\rule{.1in}{.1in}}

\makeatletter
\def\@seccntformat#1{\csname the#1\endcsname.\quad}
\makeatother

\newtheorem{theorem}{Theorem}

\newtheorem{corollary}{Corollary}

\newtheorem{definition}{\noindent Definition}
\newtheorem{example}{Example}

\newtheorem{lemma}{Lemma}

\newtheorem{proposition}{Proposition}
\newtheorem{remark}{Remark}

\renewenvironment{proof}[1][Proof]{\noindent\textbf{#1.} }{\ \rule{0.5em}{0.5em}}

\begin{document}

\title{{\Large \textbf{Scale Free Analysis and  the Prime Number Theorem}}}
\author{  Dhurjati Prasad Datta\thanks{%
Corresponding author, email: \ dp${_-}$datta@yahoo.com}
and Anuja Roy Choudhuri\thanks{
Ananda Chandra College, Jalpaiguri-735101, India, email: anujaraychaudhuri@ymail.com} \\
Department of Mathematics, University of North Bengal \\
Siliguri,West Bengal, Pin: 734013, India}
\date{}
\maketitle

\begin{abstract}
We present an elementary proof of the prime number theorem. The relative error  follows a golden ratio scaling law and respects the bound obtained from the Riemann's hypothesis. The proof is derived in the framework of a scale free  nonarchimedean extension of the real number system exploiting the concept of relative infinitesimals introduced recently in connection with ultrametric models of Cantor sets. The extended real number system is realized as a completion of the field of rational numbers $Q$ under a {\em new} nonarchimedean absolute value, which treats arbitrarily small and large numbers separately from a finite real number.
\end{abstract}

\begin{center} 
\end{center}

\begin{center} 
{\bf AMS Classification Numbers:} 11A41, 26E30, 26E35, 28A80
\end{center}


\baselineskip=19pt

\newpage

\section{Introduction}

We present a new proof of the Prime Number Theorem \cite{ed}. We call it elementary because the proof does not require any advanced techniques from the analytic number theory and complex analysis. Although the level of presentation is truly {\em  elementary} even in the standard of  the elementary calculus, except for some basic properties of non-archimedean spaces \cite{na1,na2}, some of the novel analytic structures that have been  uncovered here seem to have significance not only in number theory but also in other areas of mathematics, for instance the noncommutative geometry \cite{ncg}, infinite trees \cite{tree} and networks \cite{netw}.

The proof of the PNT is derived on a scale invariant, nonarchimedean model {\bf R} of the real number system $R$, involving nontrivial infinitesimals and infinities. The model $\bf R$ is realized as a completion of the field of rational numbers $Q$ under a {\em new} nonarchimedean absolute value $||.||$, which treats arbitrarily small and large numbers separately from any finite number. Stated more precisely, the absolute value reduces to the usual Euclidean value for a finite rational (real) number, whereas for an arbitrarily small number, defined in a limiting sense, it assigns a new scale invariant value. In the framework of the ordinary real analysis, such a value would be identified simply as zero. Arbitrarily large numbers are also assigned, by inversion,  analogous finite scale invariant values. The new absolute value is thus intrinsically scale invariant and inversion symmetric. The model constructed here is distinct from the usual nonstandard models of $R$ in two ways: (i) infinitesimals arise because of our nontrivial scale invariant treatments of small and large elements and so may be regarded members of $R$ itself (analogous to the spirit of Nelson \cite{ns1} (recall that Robinson's \cite{ns2} infinitesimals are {\em new} elements added to $R$ and may be considered as extraneous to $R$)) and (ii) It is a completion of $Q$ under the new absolute value. Thanks to the Ostrowski's theorem\cite{na1}, the so called scale invariant, infinitesimals are therefore modeled as $p$-adic integers $\tau_i$ with $|\tau_i |_p<1$, $|.|_p$ being the $p-$adic absolute value and is given by the adelic formula $\tau=\tau_p \underset{q>p}{\prod}(1+\tau_q)$. By inversion, infinities are identified with a general $p-$adic number $\tau$ with $|\tau|_p>1$. The infinitesimals considered here are said to be active as the definition involves an asymptotic limit of the form $t\rightarrow 0^+$, thereby letting an infinitesimal {\em directed}. We show that as a consequence the {\em value} of a scale invariant  infinitesimal $\tau$ would undergo infinitely slow variations over $p-$adic fields $Q_p$ as a scale free real variable $t^{-1}$, called the internal (intrinsic) timelike variable, approaches $\infty$ through the sequence of primes $p$. We show how these $p-$adic infinitesimals leaving in $\bf R$ conspire, via  nontrivial absolute values, to {\em influence} the structure of the ordinary real number set $R$, so that a finite real number $r$ gets an {\em infinitely small} correction term given by $r_{\rm cor}=r + \epsilon(t)||\tau||$, where $\epsilon(t^{-1})=\frac{\log t^{-1}}{t^{-1}}$ is the inverse of the asymptotic PNT formula of the prime counting function $\Pi(t^{-1})= \underset{p<t^{-1}} {\sum} 1$. In the ordinary analysis, there is {\em no room} for such an $\epsilon$ thus making the value of $r$ exact, viz., $r_{\rm cor}=r$. To recall \cite{ed}, 

{\bf Prime Number Theorem} states that  $\Pi(t^{-1})\epsilon(t^{-1})=1$ as $t^{-1}\rightarrow \infty$. Moreover, according to the Riemann's hypothesis, the relative correction (error) should be given by  $\Pi(t^{-1})\epsilon(t^{-1})=1+O(t^{{\frac{1}{2}}-\sigma})$, for any $\sigma>0$. So far no proof of the PNT could retrieve and substantiate the RH correction term, although all the current experimental searches on primes are known to agree with the RH value.

The proof of the PNT in the present formulation is accomplished by proving that the value $||\tau||$ of a scale free infinitesimal actually corresponds to the prime counting function
$\Pi(t^{-1})$ as the internal time $t^{-1}$ approaches infinity through larger and larger scales denoted by primes $p$. To this end we consider an equivalent (infinite dimensional) extension $\Re$ of $R$, in the usual metric topology, in which increments of a variable are mediated by a combination of linear translations and inversions. We show that there exist two types of inversions, viz., the {\em global or growing mode} leading to an asymptotic finite order variation in the value of a dynamic variable of $\Re$ following the asymptotic growth formula  of the prime counting function. On the other hand, the {\em localized} inversion mode is shown to lead to an asymptotic (golden ratio) scaling to a directed (dynamic) infinitesimal and the relative correction to the PNT.

\section{Motivation}

The motivation of the work is the recently uncovered nonsmooth solutions of the scale free equation \cite{dp1,dp2}
\begin{equation}\label{sfe}
t{\frac{d\tau}{dt}}=\tau, \ \tau(1)=1
\end{equation}
\noindent Although paradoxical (in view of the Picard's uniqueness theorem), it is difficult to ignore such solutions as pathological and the present work is an attempt to place them in a rigorous footing. To recapitulate the construction of such a solution, we introduce a new iteration in the neighbourhood of $t=1$ as follows. Let $t_\pm=1\pm\eta$ and $\tau_\pm=\tau(t_\pm)$. The standard solution is thus given by $\tau_{s\pm}=t_\pm$. We write a new solution as \cite{dp2} 
\begin{equation}\label{ns}
\tau_+=t_+, \ \tau_-=(1/t_+)\tau_1
\end{equation}
\noindent where the correction factor $\tau_1$ satisfies the self-similar equation 
\begin{equation}\label{sse}
t_{1-}{\frac{d\tau_{1-}}{dt_{1-}}}=\tau_{1-}
\end{equation}
\noindent corresponding to the smaller scale variable $t_{1-}=1-\eta^2$. Continuing this iteration ad infinitum we retrieve the standard solution as $\tau_-=\underset{i}\prod (1/t_{i+})=1-\eta,$ where $t_{0+}=t_+$.  We, however, get a second order discontinuous solution if we distort the small scale variables by introducing a parity violating (residual) rescaling $\tilde t_{1-}=\alpha_1 t_{1-}=1-\eta_1$,( but $\tilde t_{1+}= 1+\eta_1\neq\alpha_1 t_{1+}$) where $\alpha_1=1+\epsilon_1$ and $\eta_1=\alpha_1(\eta^2-\epsilon_1/\alpha_1)$. Then one obtains a solution as 
\begin{equation}\label{ns2}
\tau_+=t_+, \ \tau_-=({\frac{C}{t_+\tilde t_{1+}}})\tau_{2-}, \ \tilde t_{1+}=1+\eta_{1}
\end{equation}
\noindent where $C$ is a constant, so that the second order correction satisfies the self similar eq(\ref{sse}) in the rescaled variable $\tilde t_{2-}=1-\eta_2, \ \eta_2=\alpha_2(\eta_1^2- \epsilon_2 /\alpha_2)$, and so on ad infinitum. We note that $\tilde t_{n-}\rightarrow 1$ as $n\rightarrow \infty$, very fast because of two effects reducing small scale variables $\eta_n$, viz., (i) quadratically and also (ii) due to actions of residual rescalings at every iteration. We call the new iterated solution eq(\ref{ns2}) a generalized (nonsmooth) solution of eq(\ref{sfe}) since it has a discontinuity in the second derivative at $t=1$ for nonzero scaling parameters $\epsilon_n$. One can generate higher ($2^n$ order) derivative discontinuous solutions by initiating residual rescalings at an appropriate level, viz., at $(2^n-1)$th order iterated self similar equation. The smooth (standard) solution is thus retrieved in a non trivial manner, viz., by postponing the application of residual rescalings upto an infinite order of iterations. Another important property of nonsmooth solutions is the absence of the reflection symmetry. Let 
$P: P t_\pm =t_\mp$ denote the reflection transformation near $t=1$ 
($P\eta=-\eta$ near $\eta=0$). Clearly, eq(\ref{sfe}) is parity symmetric. 
So is the standard solution $\tau_{\pm} = t_\pm$ (since $P \tau_s = 
\tau_s$). However, the solution (\ref{ns2}) breaks this discrete symmetry 
spontaneously: $\tau_-^P=P \tau_+ = t_-,\, \tau_+^P = P \tau_- = 
C{\frac{1}{t_-}} {\frac{1}{t_{1+}^\prime}} 
{\frac{1}{t_{2+}^{\prime}}}\ldots$, which is of course a solution of 
eq(\ref{sfe}), but clearly differs from the original solution, $\tau_\pm^P \neq \tau_\pm$. Finally, the class of nonsmooth solutions could further be extended if one introduces nontrivial iterations even at right hand branch $\tau_+$ of the solution (\ref{ns2}) with a different  set of residual rescalings $\tilde\alpha_n$. Incidentally, the set of rescalings must be infinite. One can verify that the standard solution is retrieved for a finite set of scaling parameters.

In view of the above (rather large) class of nonsmooth solutions one feels compelled to consider a nonarchimedean framework  of the real analysis for a rigorous justification, which we intend to do here. We show that the said nonsmooth class of solutions actually correspond to the smooth solutions of eq(\ref{sfe}) in the associated nonarchimedean extension $\bf R$.

\section{Non-archimedean model}

\subsection{Infinitesimals}

Let $\tilde R$ be a nonstandard extension \cite{ns2} of the real number set $R$. Then an element of $\tilde R$, denoted as $\bf t$, is written as ${\bf t}=t+\bf 0$, $t\in R$, where $\bf 0$ denotes the set of infinitesimals (monad).  The set $\bf 0$ and hence $\tilde R$ is linearly ordered that matches with the ordering of $R$. The set $\bf 0$ is thus of  cardinality $c$, the continuum. The nonzero elements of $\bf 0$ are new numbers added to $R$ which are constructed from the ring ${\bf S}$ of sequences of real numbers via a choice of an ultrafilter to remove the zero divisors of $\bf S$. A nonstandard infinitesimal is realized as an equivalence class of sequences under the ultrafilter (for a recent illuminating discussion of ultrafilters and nonstandard models see \cite{tao}) and as remarked already may be considered extraneous to $R$. The magnitude of an element $\bf r$ of $\tilde R$ is evaluated using the usual Euclidean absolute value $|\bf r|_e$.

We now give a construction relating infinitesimals to arbitrarily small elements of $R$ in a more intrinsic manner. The words ``arbitrarily small elements" are made precise in a limiting sense in relation to a scale. The infinitesimals so defined are called relative infinitesimals \cite{sd1,sd2}.

\begin{definition}
Given an arbitrarily small positive {\em real} variable $t\rightarrow 0^+$, there exist a rational number $\delta>0$ and a positive {\em relative} infinitesimal $\tilde t(t)=\tilde t(t, \lambda)$, relative to the scale $\delta$,  satisfying $0<\tilde t(t)<\delta<t$ and the inversion rule $\tilde t(t)/\delta=\lambda(\delta/t)$, where $0<\lambda\ll 1$, is a real constant, so that $\tilde t$ is a (smooth) solution of  
\begin{equation}\label{sfe2}
 t{\frac{{\rm d}\tilde t}{{\rm d}t}}=-\tilde t
\end{equation}

\noindent A necessary condition for relative infinitesimals is that $0<\tilde t_1<\tilde t_2<\delta$ means  $0<\tilde t_1+\tilde t_2<\delta$ (so that $\tilde t_i$s, $ i=1,2$ must be sufficiently small relative to $\delta$). A {\em relative} infinitesimal $\tilde t$ is negative if $-\tilde t$ is a positive relative infinitesimal.
\end{definition}

Now, because of the linear ordering of ${\bf 0}^+$, the set of positive infinitesimals of $\tilde R$, as inherited from $R$ and the fact that the cardinality of ${\bf 0}^+$ equals that of $R$, there is a one-one correspondence between ${\bf 0}^+$ and $(0,\delta)\subset R$ which, we write in the present context as $\tau(t)=\tau_0(\delta/t)$ for an infinitesimal $\tau_0\in \bf 0^+$ and a relative infinitesimal $0<\tilde t<\delta, \ \delta\rightarrow 0^+$, which relates to the {\em real} variable $t$ by the inversion rule (\ref{sfe2}). Indeed, given $\delta>0$ in $R$, there exists a positive infinitesimal $\tau_0\in \bf 0^+$ such that $\tau/\tau_0=\tilde t/(\lambda\delta)=\delta/t$.  We, henceforth, identify $\bf 0^+$ (the set of positive infinitesimals) with the set of relative infinitesimals in $I_{\delta}^+=(0,\delta)\subset R$. The qualifier relative will also be dropped. We remark that, in this framework, a {\em real} variable $t$ is defined relative to the scale $\delta$ by the condition $t>\delta$.
Infinitesimals , so modeled, will be assigned a new absolute value. The real number set  $R$ equipped with this absolute value (denoted henceforth by $\bf R$) will be shown to support naturally the generalized class of solutions of eq(\ref{sfe}).

\begin{definition}
   An (relative) infinitesimal $\tilde t\in I_{\delta}=(-\delta,\delta) \ (\neq 0)$ is  assigned with a new absolute value,  $|\tilde t|=\log_{\delta^{-1}}\tilde t_1^{-1}, \ \tilde t_1=|\tilde t|_e/\delta$ as $\delta \rightarrow 0^+$. We also set $|0|=0$.
\end{definition}

\begin{remark}
{\rm 
We observe that there exists a nontrivial class of infinitesimals (viz., those satisfying $|\tilde t|_e\leq \delta.\delta^{\delta}$) for which the value $|\tilde t|$ assigned to an infinitesimal $\tilde t$ is a real number, i.e., $|\tilde t|\geq \delta$. One of our aim here is to point out nontrivial influence of these infinitesimals in real analysis. This is to be contrasted with the conventional approach. The Euclidean value of an infinitesimal is  numerically an infinitesimal. Further, the limit $ \delta\rightarrow 0^+$ is, of course, considered in the Euclidean metric. 

We also notice that the inversion in Definition 1 is nontrivial in the sense that in the absence of it, the scale $\delta$ can be chosen arbitrarily close to an infinitesimal $\tilde t$ (say), so that letting $\delta\rightarrow \tilde t$, which, in turn, $\rightarrow 0^+$, one obtains $|\tilde t|=0$. Thus, dropping the inversion rule, we reproduce the ordinary real number system $R$ with zero being the only infinitesimal (c.f., Erratum of \cite{sd1}).}
\end{remark}

Clearly, the above absolute value is well defined and also scale invariant. For, even as $\delta\rightarrow 0^+$ the relative ratio $\eta=t/\delta$ might be a constant (or approaches zero at a slower rate) in (0,1). Further, an infinitesimal $\tau\in \tilde R$ has a countable number of different realizations (representations), each for a specific choice of the scale $\delta$, having valuations $|\tilde t|_{\delta}$, where $\tilde t\in I_{\delta}$,  the image of $\tau$ under the above correspondence. Indeed, given a decreasing sequence of ({\em primary}) scales $\delta_n$ so that $\delta_n\rightarrow 0$ as $n\rightarrow \infty$, the limit in the above definition could instead be evaluated over a sequence of {\em secondary} smaller scales of the form $\delta_n^m, \ m \rightarrow \infty$ for each fixed $n$.  This observation allows one to extend the Definition 2 slightly which is now restated as 
\begin{definition}
(i)  Infinitesimals $\tau=\tilde t_n/\delta^n$ satisfying $0<\tilde t_n<\delta^n$ (and $\tilde t_n/\delta^n=(\tilde t/\delta)^n$ ) as $n\rightarrow \infty$ are called (positive) scale free $\delta-$ infinitesimals. By inversion, elements of $|\tau|_e>1$ are called scale free $\delta-$ infinities. In this scale free notation, all the finite real numbers are mapped to 1. We denote this set of $\delta$ infinitesimals and infinities by $R_{\delta}$.

 (ii)  A relative ($\delta$) infinitesimal $\tau(\neq 0)\in I_{\delta} $ is  assigned with a new ($\delta$ dependent) absolute value,  $|\tau|_{\delta}=\log_{\delta^{-n}}\tau_1^{-1}, \ \tau_1=|\tilde t_n|_e/ \delta^n, $ as $n \rightarrow \infty$. 
 \end{definition}
\noindent The Euclidean absolute value, however, is uniquely defined $|t|_e=t, \ t>0$.

\begin{proposition}
$|.|_{\delta}$ defines a nonarchimedean semi-norm on $\bf 0$.
\end{proposition}

\begin{remark}{\rm
To simplify notations, $|.|_{\delta}$, is written often as $|.|$. 
By a semi-norm we mean that $|.|$ satisfies  three properties (i) $|\tau|>0,\ \tau\neq 0$, (ii) $|-\tau|=|\tau|$ and (iii) $|\tau_1+\tau_2|\leq {\rm max}\{|\tau_1|, |\tau_2|\}$. Property $(iii)$ is called the strong $($ultrametric$)$ triangle inequality \cite{na1}.}
\end{remark}

\begin{proof}
 The first two are obvious from the definition. For the third, let $0<\tau_2<\tau_1$ in $\bf 0$. Then there exists $\delta>0$ and $\tilde t_i>0$ so that $0<\eta_2<\eta_1<1$ where $\eta_i=\tilde t_i/\delta^n\neq 1$ and $|\tau_i|=\log_{\delta^{-n}}\eta_i^{-1}$, since by definition, $\tau_i\propto \eta_i$. Clearly, $|\tau_2|>|\tau_1|$. Moreover, $0 <\eta_2<\eta_1<\eta_1+\eta_2<1$, by Definition 2. We thus have $|\tau_1+\tau_2|=\log_{\delta^{-n}}(\eta_1+\eta_2)^{-1}\leq \log_{\delta^{-n}}\eta_2^{-1}\leq |\tau_2|$.
One also has, $|\tau_1-\tau_2|= |\tau_1+(-\tau_2)|\leq {\rm max}\{|\tau_1|,|\tau_2|\}=|\tau_2|$. 
\end{proof}

Now, to restore the product rule, viz., $|\tilde t_1\tilde t_2|=|\tilde t_1||\tilde t_2|$, we note that given $\tilde t$ and $\delta$ ( $0<\tilde t<\delta$ ), there exist $0<\sigma(\delta)< 1$ and $v: {\bf 0}\rightarrow R$ such that 
\begin{equation}\label{value}
\frac{\tilde t_n}{\delta^n}= (\delta^n)^{\sigma^{v(\tilde t)}}
\end{equation}
\noindent so that, in the limit $n\rightarrow \infty$,  we have $|\tilde t|= \sigma^{v(\tilde t)}$. For definiteness, we choose $\sigma(\delta)=\delta$ (this is justified later). The function $v(\tilde t)$ is  a (discretely valued) valuation satisfying (i) $v(\tilde t_1\tilde t_2)=v(\tilde t_1)+v(\tilde t_2)$ and (ii) $v(\tilde t_1+\tilde t_2)\geq {\rm min}\{ (v(\tilde t_1),v(\tilde t_2)\}$. As a result, $|\tilde t_1\tilde t_2|=|\tilde t_1||\tilde t_2|$ and hence

\begin{proposition}
$|.|$ defines a nonarchimedean absolute value on $\bf 0$.
\end{proposition}

\begin{remark}
{\rm The above definition of valuation (\ref{value}) can be extended further to include an extra piece, viz., 
\begin{equation}\label{value1}
\frac{\tilde t_n}{\delta^n}=(\delta^n)^{(\sigma^{v(\tilde t)} + \xi(\tilde t, \delta))}
\end{equation}
\noindent where $\xi \ >0$, vanishes with $\delta$ in such a manner that $(\delta^n)^{\xi}=1$ in the limit. This observation offers an alternative definition of a scale free ($\delta$) infinitesimal, viz., $\tau=\lim {\frac{\tilde t_n}{(\delta^n)^{(1+|\tau|_{\delta})}}}=O(\delta^{n\xi}), \ n\rightarrow \infty$, which will be useful in the following. Further, the (+) sign in eq(\ref{value1}) tells that $\tau$ actually is a nontrivial infinitesimal lying in (-1,1)$\subset R_{\delta}$}.
\end{remark}

We now recall the general topological structure  of a nonarchimedean space \cite{na1,na2}.

\begin{definition}
The set $B_{r}(a)=\{\ t\mid ||t-a||<r\ \}$ is called an open ball in $\mathbf{{0}}.\ $The set $\bar{B}_{r}(a)=\{\ t\mid ||t-a||\leq r\ \}$ is a closed ball in $\mathbf{0}.$
\end{definition}

\begin{lemma}
$(i)$ Every open ball is closed and vice-versa (clopen ball) $(ii)$ every
point $b\in B_{r}(a)$ is a centre of $B_{r}(a).\ (iii)$ Any two balls in $%
\mathbf{0}$ are either disjoint or one is contained in another. $(iv)$ $%
\mathbf{0}$ is the union of at most of a countable family of clopen
balls. $(v)$ The set $\bf 0$ equipped with the absolute value $|.|$ is totally disconnected.
\end{lemma}

The proof of these assertions follow directly from the ultrametric property. Because of the property (iv) the set $\bf 0^+$ can be covered  by at most a  countable family of clopen balls viz.,  ${\bf 0^+}=\cup B(t_i)$ where $t_i$ is a bounded sequence in {\bf 0}, on each of which the absolute value $|.|$ can have a constant value. With this choice the absolute value $|.|$ is discretely valued. As already indicated, the discreteness of the valuation group \cite{na2} also follows from the choice (\ref{value}).

\begin{remark}
{\rm To emphasize, the definition of relative infinitesimals takes note of relative position of $\tilde t$ with respect to $\delta$, which could then be extended as a geometric progression to a sequence $0<\tilde t_n<\delta^n$ so that $\tilde t_n/\delta^n=(\tilde t/\delta)^n$ for the evaluation of $|\tilde t|_{\delta}$. ( To simplify notations, we henceforth, use the same symbol $\tau$ to denote a scale free infinitesimal and the sequence $\tilde t_n$ of arbitrarily small real numbers called the real valued realizations of the infinitesimal $\tau$ ).  $\delta$ infinitesimals carry traces of residual influence of the scale, as reflected in the corresponding absolute values $|\tilde t|_{\delta}= \log_{\delta^{-1}} \delta/|\tilde t|_e$, whereas  a genuinely scale free one should be independent of any scale. We notice  that the above absolute value(s)  seemingly awards the real number system with a {\em novel } structure, viz., for an arbitrarily small scale $\delta$, numbers $t$ and $\tilde t$'s satisfying $t>\delta$ and $\tilde t< \delta$ now are represented as $t=\delta.\delta^{-|\tilde t_0|}$ and $\tilde t=\lambda.\delta.\delta^{|\tilde t_0|}$, so that the inversion rule is satisfied. Actually, $\tilde t$ belongs to a closed subinterval $F$ (say) of $(0,\delta)$, the size of which is determined by $\lambda$. Here, $\tilde t_0$ is a special reference point in $F$, for instance, $\tilde t_0=t^{-1}\in F$. It is often useful to rewrite the inversion rule as an exponentiation: $\tilde t/\delta=(\delta/t)^{\mu}$, so that $\mu \log(t/\delta)=\log\lambda^{-1} +\log(t/\delta)$, for a given $t$ and $\delta$. It also follows that although $\lambda$ is a constant, the exponent $\mu$ is actually a function both of the real variable $t$ and the scale $\delta$. For $t\rightarrow \delta$ and $\delta \rightarrow 0^+$, we have $\mu\rightarrow \infty$ and $|\tilde t_0|\rightarrow 0$ in such a manner that $|\tilde t|$ may have a finite value.  For a $\delta$ infinitesimal, on the other hand,  $\mu$ may tend to $1^+$ as $\delta\rightarrow 0^+$. Indeed, in that case, we have, for a given arbitrarily small $t$ and $\delta$, a sequence $\tilde t_n$ such that $\tilde t_n =\delta^n. \delta^{\mu|\tilde t_0|_{\delta}}= \delta^n.{\delta^n}^{(\tilde \mu|\tilde  t_0|_ {\delta})}$ where $\tilde \mu=\mu/n \rightarrow 1$ for a sufficiently large $n$. Notice that such a sequence always exists. In the limit  $\delta\rightarrow 0^+$, a $\delta$ infinitesimal should go over to a scale free infinitesimal. We note also that the nontrivial factors of the form $\delta^{-|\tilde t_0|}$ and $\delta^{\mu|\tilde t_0|}$  in the above representation correspond to the residual rescalings mentioned in Sec.2. In fact, letting $t_1=t/\delta=\delta^{-|\tilde t_0|}\approx 1+\eta$, so that $\eta=|\tilde t_0|\log\delta^{-1}$, we get $\tilde t_1= \tilde t/\delta=\delta^{\mu|\tilde t_0|}\approx 1-\mu\eta, \mu=O(1)(>0)$.
Moreover, the rate of approach of a real variable $t$, which equals 1 in the ordinary analysis, gets slower in the presence of scale free infinitesimals. In fact, $t$ approaches 0 now as $t^{1-\alpha}, \ \alpha=|\tilde t|$, rather than simply as $t$.}
\end{remark}
\begin{remark}{\rm 
Further,  the relativity of a real variable and an infinitesimal could be extended over infinitesimals. Let $0<\tilde t<\delta<t$. Then for a smaller scale $\tilde \delta$, given by $0< \tilde \delta< \tilde t<\delta<t$, $\delta$-infinitesimal $\tilde t$ behaves as a {\em real variable} when a smaller scale ($\tilde\delta)$ infinitesimal (relative to the original real variable $t$) $\tilde{\tilde t}$ is given by $0<\tilde{\tilde t}<\tilde\delta<\tilde t$ and so on ad infinitum.Thus there exist a scale dependent ordering among infinitesimals. To emphasize, a real variable undergoes changes by linear shifts whereas transition between a real variable and an infinitesimal is accomplished by {\em inversions} relative to the  scale $\delta$ as stated in Definition 1. From now on, we indicate a real variable by $t$ when an infinitesimal is denoted either by $\tilde t$ or $\tau$.}
\end{remark}

\begin{remark}{\rm 
Another point of significance to note is that the definition of relative $(\delta$) infinitesimals (and infinities) of (i) is based purely in relation to the Euclidean absolute value. The new absolute value is awarded to nontrivial infinitesimals only, viz., the elements of (-1,1) of $R_{\delta}$ (where, of course, $|0|_{\delta}=0$, by definition) and also to the elements of $R_{\delta}{\backslash}(-1,1)$, by inversion (c.f. Definition 5).}
\end{remark}

\begin{example}
{\rm To give an example of Definition 2, that is, of a $\delta$ independent absolute value, let $0<\tilde t<\delta<t$ such that $t=\delta.\delta^{-k}$ for a small but nonzero constant $k>0$, so that we have a class of infinitesimals $\tilde t=\lambda.\delta.\delta^k$, $\lambda>0$. Then $|\tilde t|=\underset{\delta\rightarrow 0}{\lim}\log_{\delta^{-1}} (\delta/\tilde t)=k,$ since $\lim \log_{\delta^{-1}}\lambda=0 $. Thus, all the elements from a closed subinterval $F$ of $(0,\delta)$ are awarded the same absolute   value $k$.}  This also shows that the Definition 2 is well defined and not empty.
\end{example}

\begin{example}{\rm 
To determine how $\lambda$ decides the size of the closed interval $F$, we consider an example as motivated by the triadic Cantor set. Let $0< \tilde t <1/3< \tilde t_0<1<t$ such that $\tilde t_0=t^{-1}$ and $\tilde t=\lambda \tilde t_0$. Thus, for a given $t$, $\tilde t\in [0,1/3]$ when $0\leq\lambda\leq t/3$. The value of $\mu$ is given by $\mu=\frac{\log\lambda^{-1}+\log t}{\log t}$ where $\tilde t=t^{-\mu}$. For a fixed $\lambda$, $\mu \rightarrow \infty$ as $t\rightarrow 1^+$. For an associated $\delta$ infinitesimal $\tilde t_n$, we have, on the other hand, $\tilde\mu= \mu/n$, so that  $\tilde \mu=O(1)$ as $n \rightarrow \infty$, when $t\approx \lambda^{-1/n}$. We note that although $\lambda=1$ implies $\mu=1$, existence of $\delta$ infinitesimals allow one to realise the limit $\mu\rightarrow 1$ in a nontrivial manner.}

\end{example}
\begin{example}{\rm 
The scale $\delta$ might correspond to the accuracy level in a computational problem. In this context, 0 in $R$ is identified with the interval $I_{\delta}$ (having cardinality of continuum) and thus is raised to $\bf 0$. A computation is therefore interpreted as an activity over an extended field $\tilde R$ (for an analogous, nevertheless, different approach, see \cite{berz}). By letting $\delta^n\rightarrow 0^+$ as $n\rightarrow \infty$, we consider an infinite precision computation, which is achieved progressively by increasing the accuracy level, when real numbers are represented $\delta$-adically, for instance, the binary or decimal representation corresponding to $\delta=1/2$ or $\delta=1/10$ respectively. In the process, one  arrives at a class of ( genuine)  infinitesimals $\tau=t/\delta^n\in (-1,1), \ n\rightarrow \infty$, called the scale free $\delta-$infinitesimals here, which seem to remain available (meaningful) even in the ordinary real analysis.  The conventional treatments of real analysis, however, are immune to such infinitesimals. To avoid any conflict with the standard real analysis results (for instance, the Lebesgue measure of $R$), the scale free infinitesimals may be assumed to live  in a zero measure Cantor set.  As a consequence, the Lebesgue measure of $R_{\delta}$ is zero. To re-emphasize, a scale free infinitesimal $\tau$ is an element of a Cantor set $C_{\delta}$ (with the scale factor $\delta$) in $I_{\delta}$, while the sequence of realizations $\tilde t_n$ corresponds to its $n$th iteration realization. In such a realization, $\tilde t_n$ (say) is an element of a closed (undeleted) subset $F_{in}$ of $I_{\delta}$, each element of which is mapped to the finite real number $t$ by the inversion rule. The Cantor set structure of the scale free infinitesimals is consistent with the ultramtricity of $\bf 0$. It is shown in Ref.\cite{sd1,sd2} that the ultrametric induced by $|.|$ can indeed lead to the usual (ultra-)metric topology of a given Cantor set.} 
\end{example}

\begin{remark}
{\bf Relationship between $|\tau|$ and $|\tau|_{\delta}$}: {\rm  As it is already noticed in remark 4, for an arbitrarily small $\delta\rightarrow 0^+$, we have the asymptotic representations $t=\delta.\delta^{-|\tilde t_0|}$ and $\tilde t=\lambda\delta.\delta^{|\tilde t_0|}$. Accordingly, $|\tilde t|$ may have a finite value. For  a $\delta$ infinitesimal $\tilde t$, on the other hand, the analogous representations are   $t=\delta.(\delta)^{-|\tilde t_0|_{\delta}}$ and $\tilde t_n= \delta^n. (\delta^n)^ { \tilde\mu(\delta)|\tilde t_0|_{\delta}}$, where $n\rightarrow \infty$ but $\delta$ is kept fixed, and  $\tilde\mu$ now depends on $\delta$. It thus follows that $|\tilde t|_{\delta}=\underset {n\rightarrow \infty}{\lim}\log_{\delta^{-n}}(\delta/\tilde t)^n = \log_{\delta^{-1}} (\delta/\tilde t)=\tilde\mu(\delta)|\tilde t_0|_{\delta}$. Recalling that $\tilde\mu= 1+\sigma_{\delta}(t)$, where $ \sigma_{\delta} \rightarrow 0$ with $\delta$, we therefore write, $|\tilde t|_{\delta}=(1+\sigma_{\delta}(t))|\tilde t_0|_{\delta}=\tilde t_1$.}
\end{remark}
\begin{example}
{\rm 
As an example of $\delta$ infinitesimals we consider $p-$ infinitesimals, which are related to the $p-$adic integers in $Z_p\subset Q_p$.
Let $\delta=1/p, p$ being a prime. Then there exist $p$ infinitesimals $\tau_p$ (actually an equivalence class of such infinitesimals) which are ordered according to the primes. Let $t=p^{-(1-1/p^r)}$, for some integer $r$, be a given value of a real variable $t$ relative to the scale $1/p$. Then we have a class of $p$ infinitesimals $0<\tau_{p} <1/p$ given by $\tau_{p} = p^{-n\mu_p(t)(1+1/p^r)}$, where, $0<r\leq n$ and $\mu_p=(1+\sigma_p),  \ \sigma_p$ being a small positive variable and goes to zero faster than $1/p$. Then we have $|\tau_{p}|=p^{-r} (1+\sigma_p)$.   The valuation $v(\tau_p)$ is now obtained as $v(\tau_p)=r+\log_p(1+\sigma_p(\tau_p))$. When $\sigma_p=0$, i.e., when $\delta\rightarrow 0$, one obtains  a $p-$ adic integer, realized as a $p$ infinitesimal, because in that case we have $|\tau_{rp}|= p^{-r}$,  corresponding to  $p-$adic integers $\tau_r\in Z_p$ so that $|\tau_r|_p=p^{-r}$. We recall that a $p-$adic integer is given by $\tau_r=p^r(1+\sum a_ip^i)$, where $a_i$ assumes values from $0,1,2,\ldots (p-1)$. The sequence of partial sums $s_m=p^r(1+\sum_0^m a_ip^i)$, which is divergent in the usual metric of $Q$ and is an infinitely large element in the conventional nonstandard models of $Q$, is realized in the present model as a $p$ infinitesimal $\tau_p$ (c.f. Sec.3.2).}
\end{example}

\begin{lemma}
A closed ball in $\bf 0$ is both complete and compact.
\end{lemma}
 The proof follows from the fact the scale free infinitesimals of $\bf 0$ live in a Cantor set.

{\rm Next, we extend this nonarchimedean structure of $\bf 0$ on the whole of $R$, which is already realized here as an {\em intrinsically} nonstandard extension $\tilde R$.
\begin{definition}
Let $I_{\delta}(r)=r+I_{\delta}, \ I_{\delta}=(-\delta,\delta), \ \delta>0$ for a real number $r\in R$. For a finite $r\in R$, i.e., when $r\notin I_{\delta}(0)$, we have $||r||=|r|_e=r$. For an $r\in I_{\delta}(0)$, on the other hand, we have $||r||=|r|=\underset{\delta\rightarrow 0}{\lim}\log_{\delta^{-1}} (\delta/r)=\sigma^{v(r)}$, while, for an arbitrarily large $r (\rightarrow \infty)$, i.e., when $|r|_e>N, \ N>0$, we define $||r||=|r^{-1}|$ with a scale $\delta\leq 1/N$.
\end{definition}

\begin{proposition}
$||.||$ is a nonarchimedean absolute value on $\tilde R$ (i.e. on $R$). It is discretely valued over the set of infinitely small and large numbers. The space $\{R,||.||\}$ is denoted as $\bf R$.
\end{proposition} 
\begin{proof}
For an infinitely small or large $r$, the proof follows from Proposition 2. For a finite (non-zero) value of $r\in R$, we have, on the other hand, $r=s+\tau(t), \ s=r-t, \tau(t)=t, \ t\in I_{\delta}^+$, so that $||r||={\rm max}\{||s||, ||\tau(t)||\}= s=|r|_e$, by letting $t\rightarrow 0$. Discreteness on the set of infinitesimals and infinities follows from the discreteness of $|.|$.
\end{proof}

\begin{proposition}
$\bf R$ is a locally compact complete (ultra-)metric space.
\end{proposition}

The proof follows from Lemma 2 and Proposition 1.

\begin{proposition}
The topology induced by $||.||$ on $R$ is equivalent to the usual topology. Moreover, the imbedding $i: R\rightarrow {\bf R}$ is continuous.
\end{proposition}

Proof is an  obvious application of Definition 5.

\begin{definition}
A ($\delta$) infinitesimal $\tau$ is an (non-archimedean) integer if $|\tau|<1$. It is a unit if $|\tau|$=1.
\end{definition}

\begin{lemma}
A unit $\tau_u$ has the form $\tau_u=1+\tau$ where $|\tau|<1$.
\end{lemma}
\begin{proof}
A scale free ($\delta$) unit is defined by $||\tau_u||=1$. According to the valuation equation (\ref{value1}), we have 
\begin{equation}\label{valueu}
{\frac{\tilde t_{un}}{\delta^n}}=(\delta^n)^{1+\xi(\tau_u,\delta)}
\end{equation}
\noindent since $v(\tau_u)=0$ ($\tilde t_{un}$ are various realizations of $\tau_u$). Thus $ {\frac{\tilde t_{un}}{\delta^{2n}}}=(\delta^n)^ {\xi} \rightarrow 1, $ as $n\rightarrow \infty$ and subsequently $\delta\rightarrow 0$. Thus writing $\tilde \tau_n=\xi\log\delta^{n}$, we have the lemma, since,  as $n \rightarrow \infty$, $\tau_u=\lim {\frac{\tilde t_n}{\delta^{2n}}}=\lim e^{\tilde \tau_n}=1-\tau_0$, when we have $||\tilde\tau(1+ O(\tilde\tau)|| =||\tau_0||$,( because of the ultrametricity of $||.||$) for a $\tau_0=\tilde\tau(1+ O(\tilde\tau))$.
\end{proof}
\begin{lemma}
Let $\tau_i$ be two $\delta_i$ infinitesimals, $i=1,2$ such that $\delta_1>\delta_2$. Then there is a canonical decomposition $\tau_1=\tilde\tau_1(1+\tau_2))$ where $|\tau_1|=|\tilde\tau_1|<1$.  
\end{lemma}
\begin{proof}
Recall that $\delta-$ infinitesimals live in (0,1), which is covered by clopen balls $B(\tilde\tau_{1i}), \ i=1,2,\ldots$. Let $\tau_1\in B(\tilde\tau_{1i})$ for some $i$. 
Suppose $\tau_{1i}$ be defined by (\ref{value}) ( that is, corresponding to $\lambda=1$ in the inversion rule). For a general infinitesimal $\tau_1$ we have the extended definition given  by (\ref{value1}), viz.,
\begin{equation}\label{value2}
\frac{\tilde t_{1n}}{\delta_1^n}=(\delta_1^n)^{(|\tau_{1i}|_{\delta_1} + \xi(\tau_1,\delta_2))}
\end{equation}
\noindent where $\xi$ goes to zero faster than $|\tau_{1i}|=\delta^{v(\tau_{1i})}$ such that $(\delta_1^n)^{\xi}=1$ as $\delta_1^n\rightarrow 0^+$. Now writing $\xi=(\tilde\xi)\frac{\log\delta_2}{\log\delta_1}$ and using Lemma 3, one obtains
\begin{equation}\label{decom}
{\frac{\tilde t_{1n}}{\delta_1^n}}
={\frac{\tilde t_{1ni}}{\delta_1^n}}{\frac{\tilde t_{un}}{\delta_2^{2n}}}
\end{equation}
and so taking the limit $n\rightarrow \infty$  the desired result follows.
\end{proof}

Let us recall that $d\tau/dt=0$ means $\tau$=constant, in $R$. However, in a nonarchimedean space $\bf R$, $\tau$ would be a {\em locally} constant function, which we call here a {\em slowly varying} function. In a nonarchimedean extension of $R$, $\lambda$ (in Definition 2) may be a slowly varying function. Thus, the nonsmooth solutions of $R$ (Section 2) are realized as smooth in the nonarchimedean space $\bf R$. We recall that differentiability in a nonarchimedean space is defined in the usual sense by simply replacing the usual Euclidean metric by the ultrametric $||x-y||, \ x,y\in \bf R$.

\begin{definition}
Let $f:{\bf R}\rightarrow {\bf R}$ be a mapping from $\bf R$ to itself. Then $f$
is differentiable at $t_{0}\in \bf R$ if $\exists $ $l\in \bf R$ such that given $%
\varepsilon >0,\exists $ $\eta >0$ so that%
\begin{equation}\label{diff1}
\left\vert \frac{\parallel f(t)-f(t_{0})\parallel }{\parallel
t-t_{0}\parallel }-\parallel l\parallel\right\vert <\varepsilon
\end{equation}%
when $0<\parallel t-t_{0}\parallel <\eta ,$ and we write $f^{\prime}(t_0)= {\frac{df(t_0)}{dt}}=l$. 

\end{definition}

\begin{remark}{\rm 
The above definition is in conformity with the more conventional definition, viz., 
\begin{equation}\label{diff2}
\parallel \frac{ f(t)-f(t_{0}) }{t-t_{0}}- l\parallel <\varepsilon
\end{equation}%
\noindent since $\left\vert \frac{\parallel f(t)-f(t_{0})\parallel }{\parallel
t-t_{0}\parallel }-\parallel l\parallel\right\vert_e <\parallel \frac{ f(t)-f(t_{0}) }{t-t_{0}}- l\parallel<\varepsilon$. 
As long as $|t-t_0|_e\rightarrow 0^+$, but  $t-t_0\geq O(\delta)$, the above definition reduces to the ordinary differentiability. But when $|t-t_0|_e\rightarrow \delta$ the above gets extended to the logarithmic derivative $t{\frac{d\log f(t_0)}{dt}}=l$, when we make use of the nonarchimedean absolute value $|.|$.}
\end{remark}

So far in the above discussion the scale $\delta$ is unspecified. In the following, we introduce the nonarchimedean absolute value (Definition 5) on the field of rational numbers $Q$, construct its Cauchy completion and finally, because of the Ostrowski theorem, relate it to the local $p$-adic fields. We get a {\em minimal} nonarchimedean extension of $R$ (and which is a subset of the above $\bf R$) for each given scale $\delta$, thus sufficing our purpose of relating the (secondary) scales $\delta$ with the inverse primes viz., $\delta=1/p$. Consequently, we have a countable number of distinct field extensions ${\bf R}_p$ of $R$ depending on the scale at which the origin 0 of $R$ is probed.

 \subsection{Completion of the field of rational numbers}
On the field of rationals $Q$, we now introduce a valuation $||.||: Q\rightarrow R_+$.
\begin{definition}
Let $I_{\delta}(r)=r+I_{\delta}, \ I_{\delta}=(-\delta,\delta), \ \delta>0$ for a rational number $r\in Q$. For a finite $r\in Q$, i.e., when $r\notin I_{\delta}(0)$, we have $||r||=|r|_e=r$. For an $r\in I_{\delta}(0)$, on the other hand, we have $||r||=|r|=\lim\log_{\delta^{-1}} (\delta/r)=\delta^{v(r)}$, while, for an arbitrarily large $r (\rightarrow \infty)$, i.e., when $|r|_e>N, \ N>0$, we define $||r||=|r^{-1}|_u$ with a scale $\delta\leq 1/N$.
\end{definition}

Proposition 2 is now restated for the rationals.

\begin{corollary}
$||.||$ is a  nonarchimedean absolute value over $Q$.
\end{corollary}

Now by the Ostrowski theorem [2], any nontrivial absolute value on $Q$ must be equivalent to any of the $p-$adic absolute value $|.|_p$, $p>1$ being a prime and $|.|_{\infty}= |.|_e$ is the usual Euclidean absolute value. From Definition 8, finite rationals of $Q$  get the Euclidean value, while $|.|$, on the arbitrarily small and large values, must be related to the $p-$adic valuation. Consequently, the set of scales are represented by the  primes $\delta=p^{-1}$. 

To construct the completion of $Q$ under $||.||$ we first consider the ring $\cal S$  of all sequences of $Q$. The zero divisors in $\cal S$ are removed by the choice of an ultrafilter, as in the usual nonstandard models of $R$ (c.f. Sec.3.1) (see \cite{nsq} for a recent work on  nonstandard extensions of Q)\footnote{The model we consider here is a subset of Robinson's model which deals with all possible sequences of real numbers. The set $\cal S$ is a ring under the formal sum and product of sequences defined by $\{x_n\}+\{y_n\}=\{x_n+y_n\}$ and $\{x_n\}\{y_n\}=\{x_ny_n\}, \ x_n,y_n\in Q$.}. The quotient set of the Cauchy sequences $\cal C (\subset \cal S)$,  under the usual absolute value, modulo the maximal ideal $\cal N$ consisting of sequences converging to 0, gives rise to the ordinary real number set $R ={\cal C}/\cal N$. The elements of diverging sequences in ${\cal S}_{\rm div}={\cal S}{\backslash}\cal C$ correspond to the infinitely large elements, when the inverse $\{a_n^{-1}\}\in {\cal S}_{\rm inv}$ of a divergent sequence $\{a_n\}$ leads to an infinitesimal in the conventional approaches of nonstandard analysis.   In our approach, this realization is, however, somewhat  reversed.

Notice that the set of divergent sequences ${\cal S}_{div}$ is  quite a large set. Now among the all possible divergent sequences, there exists a subset ${\cal S}_p$ of sequences which are nevertheless $p$(-adically) convergent (${\cal S}_{p}\subset{\cal S}_{div} $, since the sequence $\{n\}$ is $p$-adically divergent for each $p$). For each fixed $p$, let us consider the Cauchy completion of $p$ convergent sequences $\{a^p_n\}$ (say) (modulo the  sequences $p$ converging to zero), viz., the local field $Q_p$. We identify, by definition, the $p$-adic integers $\tau\in Z_p\subset Q_p$ with $|\tau|_p\leq 1$ as the $p$ infinitesimals (c.f. Example 4). On the other hand, the elements $\tilde \tau$ of $Q_p$ with $|\tilde \tau|_p>1$ are identified with infinitely large numbers of order $p$. In other words, $\tilde \tau$ denotes the $p$ limit of an inverse sequence of the form $\{(a^p_n)^{-1}\}$, leading to the inversion symmetry $\tilde \tau=\tau^{-1}$ which is valid for a suitable $p$ infinitesimal $\tau$.
The absolute value $||.||$ when restricted to ${\cal S}_p$ thus relates an  infinitesimal $\tau$ an element of $\tilde R$, a nonstandard model of $R$, to a countable number of $p-$adic realizations $\tau_p\in Z_p$ with valuations $||\tau_p||=\mu_p|\tau_p|_p , \ p=2,3,\ldots,  \ p\neq \infty$, $\mu_p$ being a constant for each $p$, as the neighbourhood of 0 in $R$ is probed deeper and deeper by letting $\delta=p^{-1}\rightarrow 0$ as $p\rightarrow \infty$( c.f Example 4). In the computational model (Example 3), this might be interpreted as (inequivalent classes of ) higher precision models of a computation. Consequently, equipped with $||.||$, the set ${\cal S}_p$ decomposes into (a countably infinite Cartesian product of) local  fields $Q_p$ in a {\em hierarchical} sense as detailed in the Lemma 5. We note that any element $\tau$ of ${\cal S}_p$ is an equivalence class of sequences of rational numbers under the chosen ultrafilter. In each of such  a class there exists  a unique sequence $\{a_n^p\}$, say, converging to a $p-$adic integer or its inverse $\tau_p$. A scale free infinitesimal  $\tau$ then relates to  $\tau_p$, and that $\tau$ indeed is an infinitesimal tells that  $||\tau||=\mu_p|\tau_p|_p\leq 1$. We say that 0 of $R$ is probed at the  depth of the (secondary) scale $1/p$  when a scale free ($\delta$) infinitesimal $\tau$ is related to a  $p-$adic infinitesimal $\tau_p$. We, henceforth, denote infinitesimals, as usual, by $\bf 0$, when $p$  infinitesimals are denoted as ${\bf 0}_p$. Thus, ${\bf 0}\equiv{\bf 0}_p$ (that is, $\bf 0$ is identified with ${\bf 0}_p$) at the (secondary) scale $\tilde \delta=1/p$.

\begin{example}
Let $a_{pn}=1+\sum_1^n\alpha_ip^i$, where $\alpha_i$ assumes values from 0,1, ...$(p-1)$. The sequence $a_{pn}$ is divergent in $R$ for each prime $p$. In the nonstandard set $\tilde R$, $\{a_{pn}\}$ denote a distinct infinitely large number for each $p$. $p-$adically, however, $a_{pn}$ converges to the unity $\tau_{pu}$ ($|\tau_{pu}|_p=1$). The scale free unity $\tau_u\in {\cal S}_p$ now denotes the {\em larger} sequence $\{a_{pn},\ \forall p\}$. At the level of secondary scale $\tilde \delta =1/p$, unity $\tau_u$ is realized as  $\tau_{pu}$. 
\end{example}

\begin{remark}{\rm 
By `hierarchical' we mean that as a scale free real variable $\tilde t=t/\delta^n, \ n\rightarrow \infty$ approaches 0 from the initial value 1 through the secondary scales $1/p, \ p\rightarrow \infty$, the ordinary real variable $t\in R$ would experience changes  over  various local fields $Q_p$ successively by inversions.}
\end{remark}

\begin{lemma}
Let $\tau_p\in Z_p$ and $\tau_q\in Z_q$, $q$ being the immediate successor of the prime $p$. Then an infinitesimal $\tau\in \bf 0$ when realized as a $(p)$  infinitesimal has the  representation
\begin{equation}\label{padic}
\tau=\tau_p(1+\tau_q)
\end{equation}
\noindent Further, a $(p)$ unit is given as $\tau_{pu}=1+\tau_p, \ |\tau_{p}|<1$. 
\end{lemma} 

\begin{proof}
Let us fix the scale at $\delta=1/p^n$. Then the proof follows from Lemma 3 and 4.
\end{proof}
\begin{corollary}
One also has the following adelic extension
\begin{equation}\label{adel}
\tau=\tau_p\underset{q>p}{\prod}(1+\tau_q)
\end{equation}
\noindent where the product is over all the primes $q$ greater than $p$.
\end{corollary}
\begin{proof}
This follows from the above lemma when the valuation formula (\ref{value1}) is extended further
\begin{equation}\label{valueadel}
\frac{\tau}{\delta^n}=(\delta^n)^{(|\tau_i| +\underset{m}{\sum}\xi_m(\tau, \delta_m))}
\end{equation}
\noindent where $|\tau|=|\tau_i|$, $\delta_m^{-1}, \ m>1$ primes greater than $p$, $\delta=1/p$ and each of the indeterminate functions $\xi_m$  satisfies conditions analogous to that in (\ref{value2}). Further, $\xi_q$ goes to zero faster than $\xi_p$ if $q>p$.
\end{proof}

Collecting together the above results, we have

\begin{theorem}
 The completion of $Q$ under the absolute value $||.||$ yields a countable number of complete scale free models ${\bf R}_p$ of $R$, such that each element ${\bf t}\in {\bf R}_p$ has the form ${\bf t}_p=t+\tau_p\underset{q>p}{\prod}(1+\tau_q), \ t\in R, \ \tau_p\in Q_p$, where $\tau_p$ is given by the asymptotic expression  $\tau_p=(p^{-n})^{(1+|\tau_p|_p(1+\sigma(\eta)))}$ where $\eta=O(\delta)$ is a real variable. Finally, ${\bf R}_p$ locally has the Cartesian product form ${\bf R}_p=R\times Q_p\times\underset{q>p}{\prod} Q_q$.
 \end{theorem}
 
\begin{proof}
The only missing element is the completeness. Let us first fix a scale $\delta=1/p^r$. Let $\{a_n\}$ be a Cauchy sequence in ${\bf R}_p$.  Then it is Cauchy either in $p-$adic metric or in the usual metric, finishing the proof of the first part. The asymptotic representation follows from Examples 1 and 3. 
\end{proof}

In the following section we discuss the nature of influences that the scale free nonarchimedean extensions of $R$ would have on the basic structure of $R$ itself. 

\section{Dynamical Properties} 

To recapitulate, we note that the set $R$, in the presence of nontrivial scales, proliferates into the above $p-$adically induced  extensions ${\bf R}_p$.  We now investigate the converse question, ``How do these field extensions  influence the standard asymptotic behaviours in $R$?"  Let us recall the standard asymptotic behaviour of an ordinary real variable $t$ as it approaches 0 is that $t$ vanishes {\it linearly} as $t \rightarrow 0$ (at the uniform rate 1). In the presence of nontrivial scales the situation is altered {\em significantly}. The point 0 is now identified with the set $\tilde I_{\delta}= [-\delta,\delta], \ \delta=1/p^r$, for some $r>0$ inhabiting infinitesimals $\tilde t$ as in Definition 1.  Corresponding to these infinitesimals there exist scale free $(p)$ infinitesimals $\tau_p =\lim \tilde t_n/\delta^n, \ n\rightarrow \infty$ with absolute values of the form $||\tau_p||= |\tau_p|=|\tau|_p(1+\sigma(\eta)) $. Here, $\eta$ is defined by the real variable $t=\delta (1+\eta)$ approaching $0^+$, viz., $\delta$ from the right. In the ordinary real analysis, $\delta=0$ and the limiting value of $t$, viz., 0 is attained {\em exactly}. For a nonzero $\delta$, this exact value is attained by the rescaled variable $t_1=t/\delta$, although the value attained is now 1 instead, i.e., $t_1=t/\delta =1$. We are thus still in the framework of the real analysis (and computational models based on this analysis). The presence of infinitesimals of the form $\tilde t$ (in the conventional sense) does not appear to induce any new structure beyond that already existing in the system. The definition of scale free infinitesimals and the associated absolute values now provide us with new inputs. 

\begin{definition}
The scale free ($p$) infinitesimals, in association with absolute values $|.|$, are called valued (scale free) infinitesimals. These infinitesimals are also said to be ``active" (directed),  when infinitesimals of conventional nonstandard models are inactive (or passive, non-directed).
\end{definition}

\begin{remark}{
Because of scale free infinitesimals, the exact equality of $R$ in the ordinary analysis is replaced by an approximate equality, for instance, the equality $t=1$ is now reinterpreted as $t=O(1)$. In an associated nonarchimedean realization ${\bf R}_p$ the exact equality is again realized, albeit, in the ultrametric absolute value viz., $||t||=1$. Further, the rescalings defined by the inversion rule in Definition 1 accommodates the residual rescalings of Sec.2 provided the absolute value of a valued infinitesimal $\tau_p$ at the scale $\delta=p^{-n}$ is given by $|\tau_p|=p^{-n}$.}
\end{remark}  

 {\bf Proof and Explanations}:
We consider infinitesimals as defined by $0<\tilde{\tilde t}<\epsilon<\tilde t<\delta<t$ ($\epsilon$ is determined shortly). We show that valued infinitesimals from $(0,\epsilon]$ would affect the ordinary value of $t$ nontrivially. Infinitesimals in $(\epsilon, \delta)$ are (relatively) passive. As stated above, in ordinary analysis, the limit $t\rightarrow 0^+$ in the presence of a scale is evaluated exactly viz., $t_1=1$ i.e., $\log t_1=0=O(\delta)$. Infinitesimals, in conventional scenario, are passive in the sense that their values remain always infinitesimally small in any linear process (dynamical problem). In the present case, however, the numerically small (in the ordinary Euclidean sense) infinitesimals lying closer to 0, are {\em dominantly valued}, so as to induce an influence over a finite real variable (numbers). This is because of the definitions  of valued infinitesimals. Indeed, we {\em reinterpret}  the valuation structure (as defined in Definitions 1-3), in the context of ordinary analysis, as one admitting an extension of the ordinary (positive) real line $(\delta,\infty)$  over $(\epsilon,\delta]\cup (\delta, \infty)$ so that ordinary zero is now identified as $[0,\epsilon],\ \epsilon = t_1^{-1}. \delta\log \delta^{-1}$ instead. As a consequence, we shall now have  $\log t_1=O(\epsilon)$.  In the language of a {\em  computational model, the accuracy of the model is increased to the level denoted by $\epsilon$}. It now follows from Remark 7, that $t.\log_{\delta^{-1}}(\lambda\delta/\tilde t)= t.|\tau_r|_p (1+\sigma_p(\eta))=O(\delta^2)$.  This follows from the inversion law (Definition 1) and the fact that $|\tilde t|\sim O(\delta)$ (Remark 1). Thus, fixing $r=n$, so that $|\tau_n|_p=p^{-n}=\delta$, we get the first correction  
\begin{equation}\label{coreqn}
(1+\eta)(1+\sigma_p(\eta))=O(1)
\end{equation}

\noindent to the value of $t_1=1+\eta$ from  $(p)$ infinitesimals, even for a fixed value of $\eta$. We note that $\tilde t_p=1+\sigma_p(\eta)$ and $\sigma_p >0$ must be of higher degree in the real variable $\eta$ as $\tilde t_p\in R_p$. Eq(\ref{coreqn}) is interpreted as one encoding the influence of the (first order) $(p)$ infinitesimals. Taking into account successively the higher order $(p)$ infinitesimals, and iterating the above steps on each rescaled variables $\tilde t_p=O(1)$, one obtains $\log(1+\sigma_p (\eta))=O(\tilde \epsilon)$, where $\tilde\epsilon = t_2^{-1}. \tilde\delta\log \tilde\delta^{-1}, \ \tilde\delta=q^{-n}$, $q$ being the immediate successor to the prime $p$, and so on , so that we get finally an extended version of the equality $t_1=O(1)\in R$ in the set $\bf R$ as
\begin{equation}\label{nd}
T(\eta):=(1+\eta)\underset{q\geq p}\prod(1+\sigma_q(\eta))(=O(1))
\end{equation}
\noindent The variable $t\in R$ is thus replaced by the modified variable $T\in \bf R$ and hence, in this extended framework, a solution of 

\begin{equation}\label{sfe3}
 t{\frac{{\rm d}\tilde t}{{\rm d}t}}=-\tilde t
\end{equation}

\noindent is written, for a $t>1$ as $0<\tilde t(t)=(T(\eta))^{-1} \ <1$, which belongs clearly  to the class of nonsmooth solutions of Sec.2. Here, $\sigma_p$'s take care of the residual rescalings,  and thereby introduce small scale variations in the value of $\eta$. These also  relate  to the size of the undeleted closed intervals at the $n$th iteration of the Cantor set accommodating $(p)$ infinitesimals. This explains the {\em loss of exact equality} in the present extension of $R$. We note that the ordinary real number set $R$ is recovered in the limit when the said closed intervals reduce to a singleton (Example 2). As a consequence, $\mu=1$, leading to $\sigma_p(\eta)=0$, thus retrieving the exact equality of $R$. For a nontrivial $\mu\neq 1$, one, however, verifies that $\sigma_p(\eta)=\alpha_p(\eta_{\tilde p}-\epsilon_p/\alpha_p)^2, \ \tilde p$ being the preceding prime, when eq(\ref{sfe3}) is solved recursively as in Sec. 2.}

\begin{remark}{\rm
It is important to note that the above derivation is performed purely in the ordinary analysis, except for the fact that we make use of the {\em special representations} of $t$ and $\tilde t$ as induced from the valued infinitesimals. Consequently, 
(i) the transitions between real and infinitesimals are interpreted as being facilitated by  inversions ( for instance, either as $t_-\longmapsto t_-^{-1}=\tilde t_+$, or as $t_+\longmapsto \tilde t_-= t_+^{-1},$  as the case may be) as oppose to linear shifts (remark 5) and (ii) the nonarchimedean $p-$adic absolute value $|\tau_r|_p$  generates a scale factor in the smaller scale logarithmic variables. In fact, this correspondence could be made more precise. 
\begin{proposition}
{\rm (Page 14, Ref\cite{na3}}) Let $ \tau_p\in Q_p, \ {\rm so \ that} \  \tau_p=p^r (1+\sum_1^{\infty} a_ip^i)$, where $a_i$ assumes values from $0,1,2,\ldots (p-1)$ and $r\in Z$. Then there exists a one to one continuous mapping $\phi: \ Q_p\rightarrow R_+$ given by $\phi(\tau_p)=p^{-r}(1+\sum a_ip^{-2i})$.
\end{proposition}

Let us denote $C_p=\phi(Q_p)$. It then follows that any bounded subset of $C_p$ is a zero (Lebesgue ) measure Cantor  set in $R$ (c.f. Example 3).
Accordingly, the treatments of ordinary analysis can be extended in a scale free manner (though remaining in the folds of the usual topology of $R$) over a more general metric space $\Re$ accommodating the above new structures. In view of Theorem 1, $\Re$ is locally (at the neighbourhood of a point $t=t_0 \in R$) a Cartesian product, viz., $\Re=R\times \underset{p}\prod C_{p}$. The product space is interpreted as an hierarchical sense. An ordinary real variable $t$ is extended over $\Re$ as $T=t\prod t_p$ where $t_p=1+\epsilon_p\tau_p$,  $\tau_p\in C_p$ being a scale free $(p)$ infinitesimal and $\epsilon_p\approx \delta p^{-1}\log (p\delta^{-1})$ denotes the enhanced level of accuracy because of valued infinitesimals at the scale $\tilde\delta=\delta p^{-1}, \ \delta\rightarrow 0^+$. Noting that $\epsilon_q/\epsilon_p=p/q\rightarrow 0$, as $\delta\rightarrow 0^+, $ for $q>p$, one may re-express  $T$ as $T=t(1+\epsilon \tau)$, where $\epsilon\approx \delta \log \delta^{-1}$ when the scale is identified with $\delta=2^{-(n-1)}$ so that $p=2$, and the scale free infinitesimal $\tau$ now resides and varies in $\prod C_p$ in an orderly (hierarchical)  manner as detailed below, as $t\rightarrow 0^+ =O(\delta)$. 
}
\end{remark}
Let us first recall (again) that the concept of relative infinitesimals is introduced originally in the context of a Cantor set \cite{sd1,sd2}. Because of these relative infinitesimals each element of the Cantor set $C_p$ is effectively identified with a closed  interval of $R_p$ (a generator of a defining IFS) at every level of the scale $\delta\rightarrow 0$. In the presence of infinitesimals a Cantor set is thus realized as $R_p$ itself except for the fact the motion in $C_p$ is now visualized as a combination of linear shift (along the closed line segment) together with an inversion in the vicinity of the Cantor point itself. As a consequence, the generalized, inversion mediated metric space $\Re$ is represented henceforth as $\Re=R\times\prod R_p$. A generalized motion in $\Re$ now is represented as follows.

\begin{definition}
The set $\Re$ is interpreted as having several branches $R$ and $R_p, \ p$ being a prime. The branches $R_p$ accommodating scale free infinitesimals and infinities are thought to be knotted at (the scale free unity) 1. A real variable $t\in R$ approaching 0=$O(\delta)$ from 1(say) is replaced by the scale free variable $t_1=t/\delta$. For simplicity of notation we continue to denote $t_1$ by $t$ and call it a scale free variable instead (which should become clear from the particular context). So, as the scale free $t\rightarrow 1^+$, the unique linear motion is replaced by two inversion mediated modes: (i) {\em Local or vertical mode}: $t_+$ is replaced by $t_+\longmapsto t_+^{-1}=\tilde t_-$ which takes note of the localized effects of infinities on an infinitesimal and (ii) {\em global or horizontal (growing) mode}:  a scale free infinitesimal $0<\eta\in R_p$ grows infinitely slowly following the linear law until it shifts to a O(1) variable living possibly in another branch $R_q$ by inversion $\eta_-\longmapsto \eta_-^{-1}=t_+=1+\tilde \eta$ where $0\lesssim \tilde\eta\in R_q$. 
\end{definition}

\begin{remark}{\rm 
The unidirectional motion of a real variable $t\in R$ approaching $0^+$ thus bifurcates into two possible modes: a variable $t$ approaching $0=O(\delta)$ from above will experience, as it were,  a bounce and so would get replaced by the inverted rescaled infinitesimal variable $\tilde t=\delta/t$ living in a scale free branch $R_p$(say). A fraction of the asymptotic limit  $t\rightarrow 0^+=O(\delta)$ of $R$ (viz., $t\longmapsto t_+=t/\delta=1+\eta, \ \eta\rightarrow 0^+$) is therefore be replaced by a growing mode of the rescaled variable $\tilde t_-^{-1}=t_{p+}=1+ \eta_p, \ t_{p+}, \eta_p\in R_p$ and $ \eta_p\gtrsim 0$ initially but subsequently growing to $\rightarrow 1^-$ in the branch $R_p$. Besides this growing mode, another fraction of the decreasing (decaying) mode of the flow of  $R$, viz., $t\longmapsto t_+=t/\delta=1+\eta, \ \eta\rightarrow 0^+$, is also available as a localized mode in another branch $R_q$ (say) in the form $t_+^{-1}=(1+ \eta_q)^{-1}=\tilde t_{q-}$, where, again $  \eta_q\approx 0$ grows to $\eta_q\rightarrow 1^-$ slowly. As a consequence, the limiting value 1 of the rescaled variable in $R$, now, gets a dynamic (multiplicative) partitioning of the form $\tilde t_{q-}t_{p+}\approx 1 \ \Rightarrow \ (1-\mu(\eta_q)\eta_q)(1+ \eta_p)=O(1), \ 0\lesssim \eta_p\in R_p, \ 0\lesssim \eta_q\in R_q$, which equivalently can also be written more conveniently as  $(1-t^{\eta_q})(1+ \eta_p)=O(1)$ (c.f. Remark 4). The local and global modes therefore appear to induce a competition between the effects generated by infinitesimals and infinities,  leading  to this dynamic partitioning of the unity. The localized factor $(1-t^{\tau})$ arising from active infinitesimals $\tau$ living in a branch of $\Re$ will lead to an asymptotic scaling of any scale free (locally constant) variable $\tilde T=T/t$ in $\Re$ (see Sec.4.3). Accordingly, $d\tilde T/dt=0$ and hence $\tau\in \prod R_p$  satisfies the scale free equation 
\begin{equation}\label{rescaled}
\log t\frac{d \tau}{d\log t}=-\tau
\end{equation}  
}
\end{remark}

Consequently, asymptotic limits either of the forms $t\rightarrow 0^+$ or $t\rightarrow \infty$, in $R$  would {\em ultimately} behave as an unidirectional (monotonically increasing ) variable in $\Re$. Moreover, as $t^{-1}\rightarrow \infty$, the ordinary linear motion of $t\rightarrow 0$ will undergo small scale {\em mutations}, because of zigzag motion of the inverted variables $\tilde t_p$'s ($\sim O(1)$), living successively in the rescaled branches $R_p$ and mapping recursively the smaller and smaller neighbourhoods of 0 closer to 1.  As a consequence, extended real numbers $T$ of $\Re$ are directed, since each of the $(p)$ infinitesimals are, by definition (as it involves a limit ), directed.

\begin{definition}
{\rm Intrinsic (Internal) Time:} A continuous monotonically increasing variable $\tilde t$ living in the product space $\prod R_p$, from the initial value 1, is called an internal evolutionary time. The rate of variations of $\tilde t$ is infinitely small because of the presence of scale factors of the form $\delta p^{-1}, \ \delta\rightarrow 0^+$.

Any variable $\tau\in \prod R_p$ is called {\em dynamic} since it undergoes spontaneous changes (mutations) relative to the (scale free) internal time $\tilde t$.
\end{definition}

With this dynamic interpretation of $\Re$, it now follows that new  solution constructed in equation (\ref{nd}) is indeed smooth in $\Re$ (as it is evident from the derivation). However, because of the presence of the irreducible O(1) correction factors this solution can not be accommodated in the ordinary analysis (i.e., even in $\Re$ ) in an exact sense. In the nonarchimedean extensions ${\bf R}_p$, such a solution is not only admissible and smooth but also exact, in the sense of absolute values, viz.,  $||\tau||=1$, since $||t||=||{\tilde t}_i||=1$ for each $i$, thus retrieving the ordinary equality $|t_1|_e=1$ in the ultrametric  sense.

We restate the above deductions as a  
                                             
\begin{lemma}
The ordinary analysis on $R$ is extended over $\Re$ with new structures as detailed above (Definition 10, Remark 12). In this extended formalism accommodating dynamic (valued)  infinitesimals, $0\in R^+$ (the set of positive reals) is identified with $[0,\epsilon],\ \epsilon = t_1^{-1}.\delta \log\delta^{-1}$, where $t_1=t/\delta$ and $t\rightarrow \delta^+$. As a consequence,   a constant in $R$ becomes a variable over valued infinitesimals of  $\Re$.
\end{lemma}

\begin{proof}
We have already seen in the above that $0\in R^+$ is actually $[0,\epsilon]$ in $\Re$. We justify it further by showing that an equation of the form 
\begin{equation}\label{const}
\frac{d\phi }{dt}=0
\end{equation}%
for finite real values of $t$ is transformed into
\begin{equation}\label{vari}
\frac{d\log\phi }{d{\log{t}}}=O(1)
\end{equation}%
for an infinitesimal $\tilde{t}$ satisfying $\frac{t}{\delta }=\lambda \
\frac{\delta}{\tilde{t}}=\delta ^{-||\tilde{t}||},\ 0<\tilde{t}%
<\delta\leq t,\ t\rightarrow 0^{+},\ \lambda >0,\ $ and $||\tilde t||={\tilde t}_1 (=|\tilde t|_p(1+\sigma(\eta))$, when one interprets
$0$ in relation to the scale $\delta $ as $O(\ \frac{\delta ^{2}}{t
}\log \delta ^{-1}\ )$. Indeed, we first notice that (\ref{const}) means $d\phi=0= O(\epsilon), \ dt\neq 0$ for an ordinary real variable. However, as $t\rightarrow \delta$, that is, as $dt=\eta\rightarrow 0 =O(\epsilon)$, the ordinary variable $t$ is replaced by the extended variable, so that $d\log_{\delta^{-1}} t=d\tilde t_1=O(\epsilon)$. As a consequence, in the infinitesimal neighrhood of $\delta$, eq(\ref{const}) is transformed into an equation of the form eq(\ref{vari}). We note that in that neighbourhood, $t$ and $\phi$, are represented as $t=\delta.\delta^{-\tilde t_1}$ and $\phi=\phi_0\delta^{k\tilde t_1}$ for a real constant $k$, whence we get  eq(\ref{vari}).
\end{proof}

\subsection{Directed tree}

Geometrically, the ordinary real number system is represented as the (undirected) real line. The ultrametric extension $\bf R$ or its equivalent inversion mediated extension (realization ) $\Re$, is, on the other hand, modeled geometrically as an infinite directed, rooted tree \cite{ncg}. This is because of the one-to-one correspondence between an ultrametric space and a wieghted, directed, rooted tree \cite{tree}. Such a wieghted, directed and rooted tree has a natural correspondence with a polygenetic or evolutionary tree showing evolutionary relationship among various entities having a common ancestor. It follows accordingly that an intrinsically evolutionary sense can naturally be attached to any ultrametric space. The concept of internal time introduced above is in consonance with the time sense induced from the evolutionary tree model. We note that the estimation and classification of all possible transitions among several branches of the tree model of $\Re$ is an interesting combinatorial problem. This and other related issues will be explored in detail separately.

Before proving the PNT in the present dynamic extension of $R$, we need two more ingredients: viz, the origin of the prime counting function and the asymptotic scaling of active infinitesimals.

\subsection {Prime Counting Function}

The prime counting function arises in connection with the {\em growing} mode of a dynamic variable in $\Re$, when we assume that a scale free variable $t$ varies over all possible primal branches in an orderly manner following {\em the order of the primes}.
 
Recall that the usual $\epsilon-\delta$ definition of limit (in $R$) does not make explicit the actual motion of a real variable $t$ approaching a fixed number, 0, say. In the present formalism infinitesimals are defined by asymptotic scaling formulas (c.f. Example 1) (which would  ordinarily correspond simply to zero), and so may be considered to carry an evolutionary arrow.  A limit of the form $t\rightarrow 0$ in $\Re$ would be interpreted in the context of a dynamical problem, so that $t^{-1}$ may be identified with the  (physical) time. More precisely, when the ordinary $R$ component of $\Re=R\times \prod R_{p}$ is free of any arrow, the nontrivial components $R_{p}$ do carry an evolutionary arrow. A  problem involving the asymptotic limit $t\rightarrow 0^+$ (equivalently $t\rightarrow \infty$) in ordinary analysis is raised in the present context over $\Re$ as the asymptotic limit of the extended variable $T=t(1+\epsilon\tau)$  where $\tau$ is an O(1) growing dynamic variable which lives hierarchically in the sets $R_p, \ p=2,3,5,\ldots$ as explained in Definition 10 and Remark 12 (a growing dynamic infinitesimal $\eta$ is represented now as $\eta=\epsilon\tau$). Thus, the ordinary limit of $T$ as $t\rightarrow 0^+$ as $\delta$, that is, $T=0$, is interpreted in the present context as   $\log(T/t)=\log (1+\epsilon\tau)=O(\epsilon\Pi(t^{-1}))$ as $t\rightarrow 0^+$ hierarchically through scales $\delta p^{-1}$. That is to say, as pointed out above, as $t\rightarrow\delta^+$, it changes over to various branches $R_p$ assuming several variables  $\tilde t_i$s (all of which are different $R_p$ valued realizations of $\tau$) having forms $\tilde t_i=1 +\eta_i, \ t_0\equiv t$ and $i$ runs over the primes. Consequently, as $\tilde t_2=(\lambda_1) t_1^{-1} \in R_2, \ t_1=t/\delta$ approaches $1/2^-$, we get the next level variable $\tilde t_3=(2\lambda_2)\tilde t_2^{-1} \in R_3$ and so on and so forth, adding one unit to the prime counting function $\Pi$ at  every change of the primal scale. Indeed, infinitesimal $\eta_i$ grows linearly (and spontaneously) to O(1) whence it undergoes inversion mediated transition of the form $\eta_{i-}\longmapsto \eta^{-1} _{i-}=1+ \eta_j$, where $j$ being the next prime.  As a consequence, we have

\begin{theorem}
The ordinary (linear) limiting behaviour of a real variable $t\rightarrow 0^+$ in the real number set $R$ is raised, in the inversion mediated set $\Re$, to the asymptotic limit of the extended variable $T=t(1 +\epsilon\tau)$, where the O(1) dynamic variable $\tau$ is realized in relation to every secondary (primal) scale $1/p$ as a variable of the form $t_p\geq 1$. As a consequence, the asymptotic limit of $T$ as $\tau (\equiv t_p)\ \rightarrow 1/p, \ p\rightarrow \infty$ is given by $\log(T/t)=O(\epsilon\Pi(t^{-1}))$, for a locally constant infinitesimal $\epsilon=O(\delta\log \delta^{-1})=O(t\log t^{-1})$, as the real variable $t^{-1}\rightarrow \infty$.
\end{theorem}

In the next section we investigate the scaling of $T(t)\in \Re$ as $t\rightarrow 0$ in $R$. 

\subsection{Scaling}

Let us begin by recalling that the main characteristic of both the inverted motions is the inherent  directed sense. That is to say, although $t_{1+}=1+\eta, \ \eta\downarrow 0^+$ in $R$, in either of the inverted motions,  we have, however, $t_{1+}=1+\tilde \eta, \ \tilde \eta\approx 0$, initially, but $\tilde \eta\uparrow 1^-$, slowly, when $ t_{1+}\in R_p$. As shown in the above sections, the growing mode induces the global evolutionary sense leading to the prime counting function. Here we study the local motion leading to the asymptotic scaling for a small scale variable $\tilde T\in \Re$.

Because of the valued infinitesimals that contribute nontrivially to the ordinary value of an arbitrarily small $t\in R$, the scaling behaviour of the corresponding extended variable $T\in\Re$ is nontrivial. As explained above, an ordinary, arbitrarily small  $t\in R$ is extended in $\Re$ as $T/t=(1-O(t^{\tau}))\phi(t^{-1})$, for a class of valued infinitesimals $\tau(t^{-1})$. Our aim here is to estimate $\lim \ \tau$ as $t\rightarrow 0$. Notice that the (-) sign in the first factor makes it a true dynamic infinitesimal living in a $R_p$ (c.f. Remark 12). The second factor $\phi$ corresponds to the growing mode of a dynamic infinitesimal and is considered in Theorem 2. 

We recall that the above limit may have a constant (nonzero)(ultrametric)  value (c.f.  Example 1). Indeed, as $t$ approaches 0, as $\delta$, the ordinary variable $t$  gets extended to the rescaled variable $\tilde T_-=T_{-}/t=1-O(t^{\tau}))$, which now approaches $0^+$, via a combination of inversions and translations. Indeed, as $t\rightarrow 0$ in $R$, $\tilde T_-$ in $\Re$ is realized as a locally constant function satisfying $d\tilde T/dt=0$ so that $\tau\in \prod R_p$ now satisfies the equation (\ref{rescaled}), i.e.
\begin{equation}\label{rescaled1}
\log t\frac{d \tau}{d\log t}=-\tau
\end{equation}
\noindent and changes from one copy of $R_p$ to another near the scale $1/p$ by inversions via a sequence of distinct realizations $\tilde t_i, \ i$ being a prime. To see in detail, let $\tilde \eta_-=t^{\tau(t^{-1})}$. As $t\rightarrow 0=O(\delta)$ linearly  and the motion should have been terminated at $\delta$ in $R$, now, instead is picked up by the rescaled variable which shifts by inversion to $\tilde \eta_{2-}=t^{2\tilde t_2}, \tilde t_2\approx 0$. The limiting motion is now transmitted over to the next generation variable $\tilde t_2\in R_2$, which grows to 1/2 linearly, until the motion is again transferred to the next level by inversion viz., $2\tilde t_2=1/(1+3\tilde t_3)$, where $ \tilde t_3(\approx 0)\in R_3$. Recall  (Definition 10) that this (and the following) local inversions essentially inject into an infinitesimal higher order influences from infinities. The new rescaled variable $\tilde t_3$ now grows to 1/3, and transmits its motion to $\tilde t_5\in R_5$ near $\tilde t_5\rightarrow 1/5^-$ by inversion, and so on successively over all the higher primal scales. The exponent in  $\tilde \eta_-$ now asymptotically assumes the form of the golden ratio continued fraction $\tilde \eta_{\infty -}=t^{\frac{1}{1+{\frac{1}{1+\ldots}}}}$, so that the exponent has the value $\nu=\frac{{\sqrt 5} -1}{2}$. As a consequence, the asymptotic scaling of a dynamic infinitesimal $\tilde T_{-}$ in $\Re$ is given as $T_{-}=t\tilde T_-\sim t^{1+\nu}, \ t\rightarrow 0^+$. As a consequence, the asymptotic small scale variations (mutations) in the dynamic infinitesimal follow a {\em generic golden ratio scaling exponent} \cite{dp1}.

 Combining this local asymptotic scaling together with the global asymptotic of Theorem 2, one finally arrives at the asymptotic law 

\begin{theorem}
The generic asymptotic behaviour of a dynamic variable $T\in \Re$, extending the ordinary real variable $t$,  is given by 

\begin{equation}\label{asymp}
\log {\frac{T}{t}}=\epsilon O(\Pi(t^{-1}))(1-O(t^{\nu}))
\end{equation} 

\noindent as $ t^{-1}\rightarrow \infty$.
\end{theorem}

The above asymptotic formula is the main result of this paper. Over any finite (time) $t$ scale, the right hand side effectively reduces to zero, recovering the standard variable $T=t$. However, in any dynamic process which persists over many (infinitely large) time scales, the correction factor may become significant leading to a finite observable correction to the evolving quantity $T=te^{O(1)}$ which may arise from the annihilation (cancellation) of the infinitesimal (locally constant variable) $\epsilon$ by the growing mode of the prime counting function. The proof of the prime number theorem now follows as a corollary to Theorem 3.

\section{Prime Number Theorem}

The locally constant infinitesimal $\epsilon(t^{-1})=O(t\log t^{-1})$ clearly corresponds to the inverse of the PNT asymptotic formula for the prime counting function $\Pi(t^{-1})$. The O(1) correction to any dynamic variable $T$ in $\Re$ is realized for a sufficiently large value of $t^{-1}$ provided (more precisely, if and only if)
\begin{equation}\label{pnt}
\epsilon(t^{-1})\Pi(t^{-1})=(1+O(t^{\nu})), \ t^{-1}\rightarrow \infty
\end{equation}
\noindent with the relative correction (error) $\Pi(t^{-1})\epsilon(t^{-1})-1=O(t^{\nu})$, which clearly respects the Riemann's hypothesis since $t^{\nu}\leq M t^{(1/2-\sigma)}$ for a suitable $M>0$ and for any $\sigma>0$, $t\rightarrow 0^+$.

\end{document}